\documentclass[a4paper]{amsart}

\usepackage[top=1.5in,right=1in,left=1in,bottom=1in,a4paper]{geometry}
\usepackage{amsmath}
\usepackage{amsthm}
\usepackage{amssymb}
\usepackage{amsfonts}
\usepackage{enumerate}
\usepackage{array}
\usepackage[all,ps,dvips,color]{xy}
\usepackage{pstricks}

\newtheoremstyle{theoremstyle}
  {10pt}      %  Space above
  {5pt}       %  Space below
  {\itshape}  %  Body font
  {}          %  Indent amount (empty = no indent, \parindent = para indent)
  {\bfseries} %  Thm head font
  {:}         %  Punctuation after thm head
  {.5em}      %  Space after thm head: " " = normal interword space;
  {}          %  Thm head spec (can be left empty, meaning `normal')

\newtheoremstyle{examplestyle}
  {10pt}      %  Space above
  {5pt}       %  Space below
  {}          %  Body font
  {}          %  Indent amount (empty = no indent, \parindent = para indent)
  {\bfseries} %  Thm head font
  {:}         %  Punctuation after thm head
  {.5em}      %  Space after thm head: " " = normal interword space;
  {}          %  Thm head spec (can be left empty, meaning `normal')

\theoremstyle{theoremstyle}
\newtheorem{theorem}{Theorem}[section]
\newtheorem*{theorem*}{Theorem}
\newtheorem{lemma}[theorem]{Lemma}
\newtheorem{proposition}[theorem]{Proposition}
\newtheorem*{proposition*}{Proposition}
\newtheorem{corollary}[theorem]{Corollary}
\newtheorem*{corollary*}{Corollary}

\theoremstyle{examplestyle}

\newtheorem{definition}[theorem]{Definition}
\newtheorem*{definition*}{Definition}
\newtheorem{remark}[theorem]{Remark}
\newtheorem*{remarks*}{Remarks}

\newtheorem*{remark*}{Remark}

\newcommand{\comment}[1]{}

\newcommand{\PBS}[1]{\let\temp=\\#1\let\\=\temp}

\newcommand{\pic}{\operatorname{Pic}}
\newcommand{\sh}[1]{\mathcal{#1}}
\newcommand{\rk}{\operatorname{rk}}

\newcommand{\Ext}{\operatorname{Ext}}
\newcommand{\End}{\operatorname{End}}

\newcommand{\Z}{\mathbb{Z}}
\newcommand{\C}{\mathbb{C}}

\newcommand{\R}{\mathbb{R}}

\newcommand{\im}{\operatorname{im}}

\begin{document}

\title[Exceptional Sequences]{Examples for exceptional sequences of invertible sheaves on rational surfaces}


\author{Markus Perling}
\address{Fakult\"at f\"ur Mathematik\\Ruhr-Universit\"at Bochum\\Universit\"atsstra\ss e 150\\44780 Bochum\\Germany}
\email{Markus.Perling@rub.de}

\date{April 2009}

\begin{abstract}
In this article we survey recent results of joint work with Lutz Hille on exceptional sequences of
invertible sheaves on rational surfaces and give examples.
\end{abstract}

\maketitle

\tableofcontents

\section{Introduction}\label{introduction}

The purpose of this note is to give a survey on the results of joint work with Lutz Hille
\cite{HillePerling08} and to provide
some explicit examples. The general problem addressed in \cite{HillePerling08}  is to
understand the derived category of coherent sheaves on an algebraic variety
(for an introduction and overview on derived categories over algebraic
varieties we refer to \cite{Huybrechts06}; see also \cite{Bridgeland06}). An important approach to
understand derived categories is to construct exceptional sequences:

\begin{definition*}[\cite{Rudakov90}]
A coherent sheaf $\sh{E}$ on $X$ an algebraic variety $X$ is called {\em exceptional}\, if
$\sh{E}$ is simple and $\Ext^i_X(\sh{E}, \sh{E})$ $= 0$ for every $i \neq 0$. A sequence
$\sh{E}_1, \dots, \sh{E}_n$ of exceptional sheaves is called an {\em exceptional}\, sequence if
$\Ext^k_X(\sh{E}_i, \sh{E}_j) = 0$ for all $k$ and for all $i > j$. If an exceptional
sequence generates $D^b(X)$, then it is called {\em full}. A {\em strongly} exceptional
sequence is an exceptional sequence such that $\Ext^k_X(\sh{E}_i, \sh{E}_j) = 0$ for all
$k > 0$ and all $i, j$.
\end{definition*}

In \cite{HillePerling08} we have considered exceptional and strongly exceptional sequences on
rational surfaces which consist of invertible sheaves. Though this seems to be a quite restrictive
setting, it still covers a large and interesting class of varieties, and the results uncover
interesting aspects of the derived category which had not been noticed so far. In particular,
it seems that toric geometry plays an unexpectedly important role.

Besides giving a survey, we also want to consider in these notes some aspects related to
noncommutative geometry, which have not been touched in \cite{HillePerling08}.
By results of Bondal \cite{Bondal90}, the existence of a full strongly
exceptional sequence implies the existence of an equivalence of categories
\begin{equation*}
\operatorname{{\bf R}Hom}(\sh{T}, \, . \,) : D^b(X) \longrightarrow D^b(A-\operatorname{mod}),
\end{equation*}
where $\sh{T}:= \bigoplus_{i = 1}^n \sh{E}_i$, which is sometimes called a tilting sheaf, and
$A := \End(\sh{T})$. This way, a strongly exceptional sequence provides a non-commutative
model for $X$. The algebra $A$ is
finite and can be described as a path algebra with relations. One of the themes of this note will
be to provide some interesting and explicit examples of such algebras.

\section{Exceptional sequences and toric systems}

Our general setting will be that $X$ is a smooth complete rational surface defined over some algebraically
closed field. We will always denote $n$ the rank of the Grothendieck group of $X$. Note that $\pic(X)$
is a free $\Z$-module of rank $n - 2$. We are interested in exceptional sequences of invertible sheaves.
That is, we are looking for divisors $E_1, \dots, E_n$ on $X$
such that the sheaves $\sh{O}(E_1), \dots, \sh{O}(E_n)$ form a (strongly) exceptional sequence.
In this situation, computation of $\Ext$-groups reduces to compute cohomologies of divisors,
i.e. for any two invertible sheaves $\mathcal{L}$, $\mathcal{M}$, there exists an isomorphism
\begin{equation*}
\Ext^k_X(\sh{L}, \sh{M}) \cong H^k(X, \sh{L}^* \otimes \sh{M})
\end{equation*}
for every $k$. So, the $\sh{O}(E_i)$ to form a (strongly) exceptional sequence, it is necessary
in addition that $H^k\big(X, \sh{O}(E_j - E_i) \big) = 0$ for every $k$ and every $i > j$ (or for
every $k > 0$
and every $i, j$, respectively). For simplifying notation, we will usually omit references to $X$
and for some divisor $D$ on $X$ we
write $h^k(D)$ instead of $\dim H^k\big(X, \sh{O}(D)\big)$. Note that for
any divisor $D$ and any blow-up $b : X' \rightarrow X$ the vector spaces $H^k\big(X, \sh{O}(D)\big)$
and $H^k\big(X', b^*\sh{O}(D)\big)$ are naturally isomorphic. This way there will be no ambiguities
on where we determine the cohomologies of $D$.

In essence, we are facing a quite peculiar problem of cohomology vanishing: constructing a strongly
exceptional
sequence of invertible sheaves is equivalent to finding a set of divisors $E_1, \dots, E_n$ such that
$h^k(E_i - E_j) = 0$ for all $k > 0$ and all $i, j$. Formulated like this, the problem looks somewhat
ill-defined, as it seems to imply a huge numerical complexity (the Picard group has rank $n - 2$ and
one has to check cohomology vanishing for $\sim n^2$ divisors). This can be tackled rather
explicitly for small examples (see \cite{HillePerling06}), but such a procedure becomes very
unwieldy in general.

Simple examples show that standard arguments
for cohomology vanishing via Kawamata-Viehweg type theorems usually are not sufficient to solve
the problem (however, see \cite{perling07a} for another kind of vanishing theorem which may be
helpful here).
So, to deal more effectively with the problem, we have to make more use of the geometry of $X$.
Recall that $\pic(X)$ (or better, $\pic(X) \otimes_\Z \R$) is endowed with a bilinear quadratic form
of signature $(1, -1, \dots, -1)$, which is given by the intersection form. Moreover, the geometry
of $X$ gives us a canonical
quadratic form on $\pic(X)$, the Euler form $\chi$, which by the Riemann-Roch theorem is of the form
$$\chi(D) = 1 + \frac{1}{2}(D^2 - K_X . D)$$ for $D \in \pic(X)$. The condition that $E_1, \dots, E_n$
yield an exceptional sequence implies that $\chi(E_j - E_i) = 0$ for every $i > j$. So, we get another
bit of information which tells us that we should look for divisors which are sitting on a quadratic
hypersurface in $\pic(X)$ which is given by the equation $\chi(-D) = 0$ for $D \in \pic(X)$. However,
looking for integral solutions of a quadratic equation still does not help very much. To improve
the situation, we have to exploit the Euler form more effectively. Considering its symmetrization
and antisymmetrization:
\begin{align*}
\chi(D) + \chi(-D) & = 2 + D^2 \quad \text{ and}\\
\chi(D) - \chi(-D) & = -K_X . D,
\end{align*}
we get immediately:

\begin{lemma}\label{acyclicitylemma}
Let $D, E \in \pic(X)$ such that $\chi(-D) = \chi(-E) = 0$, then
\begin{enumerate}[(i)]
\item\label{acyclicitylemmai}
$\chi(D) = -K_X . D = D^2 + 2$;
\item\label{acyclicitycorollaryii}
$\chi(-D - E) = 0$  iff $E . D = 1$ iff $\chi(D) + \chi(E) = \chi(D + E)$.
\end{enumerate}
\end{lemma}

Now, using Lemma \ref{acyclicitylemma}, we can bring an exceptional sequence $\sh{O}(E_1), \dots,
\sh{O}(E_n)$ into a convenient normal form. For this, we pass to the difference vectors and
set $A_i := E_{i + 1} - E_i$ for $1 \leq i < n$ and
\begin{equation*}
A_n := -K_X - \sum_{i = 1}^{n - 1} A_i.
\end{equation*}
By Lemma \ref{acyclicitylemma}, we get:
\begin{enumerate}[(i)]
\item\label{toricsystemdefi} $A_i . A_{i + 1} = 1$ for $1 \leq i \leq n$;
\item\label{toricsystemdefii} $A_i . A_j = 0$  for $i \neq j$ and $\{i, j\} \neq \{k, k + 1\}$ for some $1 \leq k \leq n$;
\item\label{toricsystemdefiii} $\sum_{i = 1}^n A_i = -K_X$.
\end{enumerate}
Note that here the indices are to be read in the cyclic sense, i.e. we identify integers modulo $n$. This leads to
the following definition:

\begin{definition}
We call a set of divisors on $X$ which satisfy conditions (\ref{toricsystemdefi}), 
(\ref{toricsystemdefii}), and (\ref{toricsystemdefiii}) above a {\em toric system}.
\end{definition}

A toric system seems to be the most efficient form to encode an exceptional sequence of invertible sheaves.
But why the name? It turns out that a toric system encodes a smooth complete rational surface.
Consider the following short exact sequence:
\begin{equation*}
0 \longrightarrow \pic(X) \overset{A}{\longrightarrow} \Z^n \longrightarrow \Z^2 \longrightarrow 0,
\end{equation*}
where $A$ maps a divisor class $D$ to the tuple $(A_1 . D, \dots, A_n . D)$. Denote $l_1, \dots, l_n$
the images of the standard basis fo $\Z^n$ in $\Z^2$. It is shown in \cite{HillePerling08} that the
$l_i$ form a cyclically ordered set of primitive elements in $\Z^2$ such that $l_i, l_i + 1$ forms a
basis for $\Z^2$ for every $i$. This is precisely the defining data for a smooth complete toric surface.
On the other hand, the construction of the $l_i$ from the $A_i$ is an example for {\em Gale duality}.
In particular, if we dualize above short exact sequence, we obtain a standard sequence from toric
geometry:
\begin{equation*}
0 \longrightarrow \Z^2 \overset{L}{\longrightarrow} \Z^n \longrightarrow \pic(Y) \longrightarrow 0,
\end{equation*}
where $Y$ denotes the toric surface associated to the toric system $A_1, \dots, A_n$ and $L$ the
matrix whose rows are the $l_i$ which act via the standard Euklidian inner product as linear forms
on $\Z^2$. By this duality, $\pic(X)$ and $\pic(Y)$ and their intersection products can canonically be
identified. So we get the following remarkable structural result:

\begin{theorem}[\cite{HillePerling08}]\label{structuretheorem}
Let $X$ be a smooth complete rational surface, let $\sh{O}_X(E_1), \dots, \sh{O}_X(E_n)$ be a
full exceptional sequence of invertible sheaves on $X$, and set $E_{n + 1} := E_1 - K_X$. Then
to this sequence there is associated in a canonical way a smooth
complete toric surface with torus invariant prime divisors $D_1, \dots, D_n$ such that
$D_i^2 + 2 = \chi(E_{i + 1} - E_i)$ for all $1 \leq i \leq n$.
\end{theorem}

This theorem at once gives us information about the possible values for $\chi(A_i)$ and the combinatorial
types of quivers associated to strongly exceptional sequences, but it does not give us a method for
constructing them. This will be discussed in the next section. We conclude with some remarks on some
ambiguities related to toric systems. Of course, for a given exceptional sequence $\sh{O}(E_1), \dots,
\sh{O}(E_n)$, the associated toric system $A_1, \dots, A_n$  is unique. However, from $A_1, \dots,
A_n$ we will get back the original sequence only up to twist, i.e. by summing up the $A_i$ we get
$\sh{O}_X, \sh{O}(A_1), \sh{O}(A_1 + A_2),\dots, \sh{O}(\sum_{i = 1}^{n - 1} A_i)$. Of course, this
does not really matter, as we are interested only in the differences among the divisors.

By construction, the condition that the cohomologies vanish a priori apply only to the $A_i$ with
$1 \leq i < n$, but it follows from Serre duality that also  $h^k(-A_n) = 0$ for all $k$. More
generally, it follows that every cyclic enumeration of the $A_i$ gives rise to an exceptional
sequence. Moreover, if $A_1, \dots, A_n$ yields an exceptional sequence, then so does
$A_n, \dots, A_1$.

The case where the $E_i$ form a strongly exceptional sequence has less symmetries in general. In
particular, we cannot expect that the higher cohomologies of $A_n$ vanish in this case. However,
if $A_1, \dots, A_n$ give rise to a strongly exceptional sequence, then also
\begin{equation*}
A_{n - 1}, \dots, A_1, A_n
\end{equation*}
does. This system then corresponds to the dual strongly exceptional sequence, given by
$\sh{O}(-E_n), \dots,$ $\sh{O}(-E_1)$.

\section{Examples of toric systems and standard augmentations}

The easiest examples of toric systems exist on $\mathbb{P}^2$ and the Hirzebruch surfaces $\mathbb{F}_a$.
We first consider the case of $\mathbb{P}^2$. Its Picard group is isomorphic to $\Z$ and generated by the
class $H$ of a line. In this case, there exists a unique toric system which is given by
\begin{equation*}
H, H, H,
\end{equation*}
which corresponds to a strongly exceptional sequence given by $\sh{O}_{\mathbb{P}^2}$, $\sh{O}_{\mathbb{P}^2}(1)$,
$\sh{O}_{\mathbb{P}^2}(2)$. The quiver with relations of the endomorphism algebra is shown in figure
\ref{p2quiver} (see also \cite{Bondal90}).
\begin{figure}[htbp]
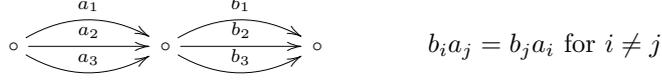

\begin{equation*}
\xygraph{!{0;<1cm,0cm>:<0cm,1cm>::}
[] *+[o]+@{o}="1"
[rr] *+[o]+@{o}="2"
[rr] *+[o]+@{o}="3"
[rrr] *\txt{$b_i a_j = b_j a_i$ for $i \neq j$}
"1":@/-1mm/"2"^(0.55){a_2}
"1":@/^1em/"2"^{a_1}
"1":@/_1em/"2"^{a_3}
"2":@/-1mm/"3"^(0.55){b_2}
"2":@/^1em/"3"^{b_1}
"2":@/_1em/"3"^{b_3}
}
\end{equation*}
\caption{The quiver associated to $\sh{O}_{\mathbb{P}^2}$, $\sh{O}_{\mathbb{P}^2}(1)$, $\sh{O}_{\mathbb{P}^2}(2)$.}
\label{p2quiver}
\end{figure}
The toric surface associated to this toric system is $\mathbb{P}^2$ again.
For the case of $\mathbb{F}_a$, we have a slightly more complicated picture. First we choose a basis $P, Q$ for
$\pic(\mathbb{F}_a)$, where $P$ is the class of a fiber of the ruling and $Q$ is chosen such that $Q - aP$ is
the class of the base of the ruling. This choice is such that $P$ and $Q$ are the two generators of the nef cone
of $\mathbb{F}_a$ and hence have no higher cohomologies. With respect to these coordinates, there exists a
unique family of toric systems which is associated to exceptional sequences:
\begin{equation*}
P, sP + Q, P, -(a + s)P + Q, \text{ where } s \in \Z.
\end{equation*}
The associated toric surface is isomorphic to $\mathbb{F}_{a + 2s}$.
The associated sequence is strongly exceptional for $s \geq -1$. For both $ s, -(a + s) \geq -1$, we obtain
{\em cyclic} strongly exceptional sequences, which we will discuss in section \ref{cyclicsection}, and which lead
to special quivers in each case. Otherwise, we will get rather uniformly looking quivers. Figure \ref{fa1} shows
the quiver with relations for For $s = -1$ and $a \geq 3$,
\begin{figure}[htbp]
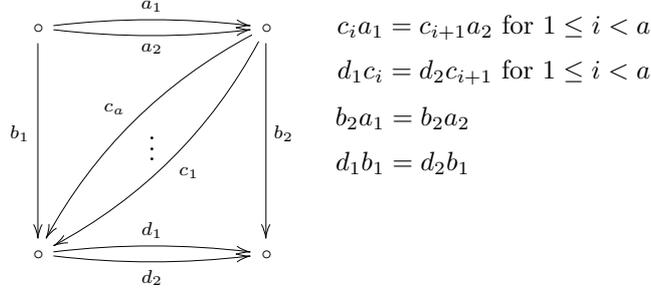

\begin{equation*}
\xygraph{!{0;<1.5cm,0cm>:<0cm,3cm>::}
[] *+[o]+@{o}="1"
[rr] *+[o]+@{o}="2"
[d(0.5)l] *+{\vdots}
[d(0.5)l] *+[o]+@{o}="3"
[rr] *+[o]+@{o}="4"
[urr] *\txt{$c_i a_1 = c_{i + 1} a_2$ for $1 \leq i < a$}
[d(0.2)] *\txt{$d_1 c_i = d_2 c_{i + 1}$ for $1 \leq i < a$}
[d(0.2)l(0.8)] *\txt{$b_2 a_1 = b_2 a_2$}
[d(0.2)] *\txt{$d_1 b_1 = d_2 b_1$}
"1":@/^0.3em/"2"^{a_1}
"1":@/_0.3em/"2"_{a_2}
"2":@/^1em/"3"^{c_1}
"2":@/_1em/"3"_{c_a}
"3":@/^0.3em/"4"^{d_1}
"3":@/_0.3em/"4"_{d_2}
"1":@/-1mm/"3"_{b_1}
"2":@/-1mm/"4"^{b_2}
}
\end{equation*}
\caption{The quiver associated to $\sh{O}_{\mathbb{F}_a}$, $\sh{O}_{\mathbb{F}_a}(P)$, $\sh{O}_{\mathbb{F}_a}(Q)$, $\sh{O}_{\mathbb{F}_a}(P + Q)$.}
\label{fa1}
\end{figure}
figure \ref{fa2} for $s \geq 0$ and $a + s \geq 1$.
\begin{figure}[htbp]
\begin{equation*}
\xygraph{!{0;<1.5cm,0cm>:<0cm,2cm>::}
[] *+[o]+@{o}="1"
[rr] *+[o]+@{o}="2"
[u(0.2)r] *+{\vdots}
[d(0.3)] *+{\vdots}
[u(0.1)r] *+[o]+@{o}="3"
[rr] *+[o]+@{o}="4"
"1":@/^0.3em/"2"^{a_1}
"1":@/_0.3em/"2"_{a_2}
"2":@/^1.5em/"3"^{b_0}
"2":@/^0.3em/"3"^(0.4){b_s}
"2":@/_0.3em/"3"_(0.4){c_0}
"2":@/_1.5em/"3"_{c_{s + a}}
"3":@/^0.3em/"4"^{d_1}
"3":@/_0.3em/"4"_{d_2}
}
\end{equation*}
\begin{align*}
b_i a_1 &= b_{i + 1} a_2 \text{ for } 0 \leq i < s  & d_1 c_i &= d_2 c_{i + 1} \text{ for } 0 \leq i < s + a\\
c_i a_1 &= c_{i + 1} a_2 \text{ for } 0 \leq i < s + a & d_1 b_i &= d_2 b_{i + 1} \text{ for } 0 \leq i < s
\end{align*}
\caption{The quiver associated to $\sh{O}_{\mathbb{F}_a}$, $\sh{O}_{\mathbb{F}_a}(P)$, $\sh{O}_{\mathbb{F}_a}\big((s + 1)P + Q\big)$,
$\sh{O}_{\mathbb{F}_a}\big((s + 2)P + Q\big)$ for $s \geq 0$ and $a + s \geq 1$.}
\label{fa2}
\end{figure}
The quivers and relations can quite directly be computed using toric methods (see \cite{perling6}, for instance). We call
toric systems on $\mathbb{P}^2$ or $\mathbb{F}_a$ {\em standard} toric systems. By the classification of rational surfaces,
we know that for any given surface $X$, there exists a sequence of blow-downs
\begin{equation*}
X = X_t \overset{b_t}{\longrightarrow} X_{t - 1} \overset{b_{t - 1}}{\longrightarrow} \cdots \overset{b_2}{\longrightarrow} X_1
\overset{b_1}{\longrightarrow} X_0,
\end{equation*}
where $X_0$ is either $\mathbb{P}^2$ or some $\mathbb{F}_a$. We assume now that one such sequence (which is far from unique,
in general) is fixed. Denote $R_1, \dots, R_t$ the classes of the total transforms on $X$ of the exceptional divisors of the blow-ups
$b_i$. Using these, we obtain a nice basis for $\pic(X)$. If $X_0 \cong \mathbb{P}^2$, we get $H, R_1, \dots, R_t$ as a basis with
the following intersection relations:
\begin{equation*}
H^2 = 1, \quad R_i^2 = -1 \text{ for every }  i, \quad R_i . R_j = 0 \text{ for every } i \neq j,  \quad
H . R_i = 0 \text{ for every }  i.
\end{equation*}
If $X_0 \cong \mathbb{F}_a$, we get $P, Q, R_1, \dots, R_t$ (where we identify $H$ and $P, Q$, respectively, with their
pull-backs in $\pic(X)$) with
\begin{gather*}
P^2 = 0, \quad Q^2 = a, \quad P . Q = 1, \quad  R_i^2 = -1 \text{ for every }  i, \\  R_i . R_j = 0 \text{ for every } i \neq j, \quad
P . R_i = Q. R_i = 0 \text{ for every }  i.
\end{gather*}
If we start with any given toric system on $X_0$, then we can successively produce new toric systems as follows. Assume that
$X_0 \cong \mathbb{P}^2$ and start with the unique toric system $H, H, H$ on $\mathbb{P}^2$.
After blowing up once, we can augment this toric system by inserting $R_1$ in any place
and subtracting $R_1$ in the two neighbouring positions, i.e., up to symmetries, we obtain a toric
system $H - R_1, R_1, H - R_1, H$ on $X_1$. Continuing with this, we essentially get two
possibilities on $X_2$, namely
\begin{align*}
& H - R_1 - R_2, R_2, R_1 - R_2, H - R_1, H\\
& H - R_1, R_1, H - R_1 - R_2, R_2, H - R_2.
\end{align*}
It is easy to see that all of these examples lead to strongly exceptional sequences for any choice of
$A_n$, with exception of the first one in the case where $b_2$ is a blow-up of an infinitesimal point,
where we necessarily have to choose the enumeration of the $A_i$ such that $A_n = R_1 - R_2$.
Similarly, if $X_0 \cong \mathbb{F}_a$, we can start with some toric system of the form
$P, sP + Q, P, -(a + s)P + Q$ and augment this sequence successively. In this case, we already have
less symmetries for the first augmentation and there are essentially two possibilities:
\begin{align*}
& P - R_1, R_1, sP + Q - R_1, P, -(a + s)P + Q\\
& P, sP + Q, P - R_1, R_1,  -(a + s)P + Q - R_1.
\end{align*}
We can augment these toric systems along the $b_i$ in the same fashion. Toric systems obtained
this way play a special role in what follows, so they need a name:

\begin{definition}
We call a toric system obtained by augmentation as above of a toric system on some $X_0$
a {\em standard augmentation}.
\end{definition}

It is straightforward to show that standard augmentations always lead to full exceptional sequences
of invertible sheaves. So, we recover the following well-known result (compare \cite{Orlov93}):

\begin{theorem}[\cite{HillePerling08}]
On every smooth complete rational surface exists a full exceptional sequence of invertible sheaves.
\end{theorem}

We have not yet determined the existence of strongly exceptional sequences. A nice class of examples
comes from simultaneous blow-ups,where we assume $X$ to be a blow-up of $X_0$ simultaneously
at several points. That is, we do not blow-up infinitesimal points. Assuming $X_0 \cong \mathbb{P}^2$,
we get:

\begin{proposition}\label{blowuponce}
The toric system
\begin{equation*}
R_t, R_{t - 1} - R_t, \dots, R_1 - R_2, H - R_1, H, H - \sum_{i = 1}^t R_i
\end{equation*}
corresponds to a strongly exceptional sequence of invertible sheaves.
\end{proposition}

\begin{proof}
A toric system which is a standard augmentation corresponds to a strongly exceptional sequence
iff $h^k(\pm \sum_{i \in I} A_i) = 0$ for every interval $I \subset \{1, \dots, n - 1\}$. So, there are
only a few types of divisors to check, namely $R_i$, $R_i - R_j$, $H - R_i$, $2H - R_i$, $H$, $2H$.
The assertion is clear for $R_i$, $H$, $2H$. A divisor of type $R_i - R_j$ has no cohomology iff
$R_i$ does not lie over $R_j$ or vice versa; this is clear, because we assumed that no infinitesimal
points are blown up. As $H$ and $2H$ are very ample, the linear systems $| H |$ and $| 2H |$ are
base-point free and we can conclude that $h^k\big(\pm (H - R_i)\big) = h^k\big(\pm (2H - R_i)\big)
= 0$ for all $k > 0$.
\end{proof}

More explicitly, the toric system of Proposition \ref{blowuponce} corresponds to the strongly
exceptional sequence
\begin{equation*}
\sh{O}_X, \sh{O}(R_t), \dots, \sh{O}(R_1), \sh{O}(H), \sh{O}(2H).
\end{equation*}
Figure \ref{p2blowup1} shows the quiver (without relations) and the fan associated to the toric
system of Proposition \ref{blowuponce} (the relations of this quiver will be described in section
\ref{deformationsection}). The black rays denote the fan associated to the toric system $H, H, H$,
the grey rays indicate the rays added to the fan by augmenting the toric system.
\begin{figure}[htbp]
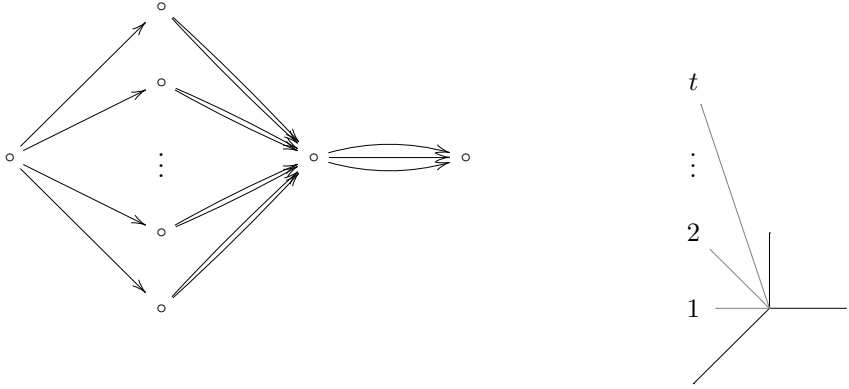

\begin{equation*}
\xygraph{!{0;<2cm,0cm>:<0cm,1cm>::}
[] *+[o]+@{o}="1"
[ruu] *+[o]+@{o}="2"
[d] *+[o]+@{o}="3"
[d] *+{\vdots}
[d] *+[o]+@{o}="4"
[d] *+[o]+@{o}="5"
[uur] *+[o]+@{o}="6"
[r] *+[o]+@{o}="7"
"1":@/-1mm/"2"
"1":@/-1mm/"3"
"1":@/-1mm/"4"
"1":@/-1mm/"5"
"2":@/^0.1em/"6"
"2":@/_0.1em/"6"
"3":@/^0.1em/"6"
"3":@/_0.1em/"6"
"4":@/^0.1em/"6"
"4":@/_0.1em/"6"
"5":@/^0.1em/"6"
"5":@/_0.1em/"6"
"6":@/^0.5em/"7"
"6":@/_0.5em/"7"
"6":@/-1mm/"7"
[rrdd]="8"
[l(0.5)d]="9"
[r(0.5)uu]="10"
[r(0.5)d]="11"
[l] *+[o]+{1}="12"
[u] *+[o]+{2}="13"
[u] *+{\vdots}
[u] *+[o]+{t}="14"
"8"-"9"
"8"-"10"
"8"-"11"
"8"-@[grey]"12"
"8"-@[grey]"13"
"8"-@[grey]"14"
}
\end{equation*}
\caption{Quiver and fan associated to a toric system coming from a simultaneous blow-up of $\mathbb{P}^2$.}
\label{p2blowup1}
\end{figure}
Similarly, one can produce a strongly exceptional sequence if $X$ can be obtained from $\mathbb{P}^2$ by
blowing up two-times in several points. That is, we can partition the blow-ups such that $b_1, \dots, b_r$
blow up points on $\mathbb{P}^2$ and $b_{r + 1}, \dots, b_t$ blow up infinitesimal points on the exceptional
divisors of the first group of blow-ups. A toric system which corresponds to a strongly exceptional sequence
then is given by
\begin{equation*}
R_r, R_{r - 1} - R_r, \dots, R_1 - R_2, H - R_1, H - R_{r + 1}, R_{r + 1} - R_{r + 2}, \dots, R_{t - 1} - R_t, R_t, H - \sum_{i = 1}^t R_i.
\end{equation*}
This can be shown similarly as in Proposition \ref{blowuponce}.
Figure \ref{p2blowup2} shows the corresponding quiver and fan.
\begin{figure}[htbp]
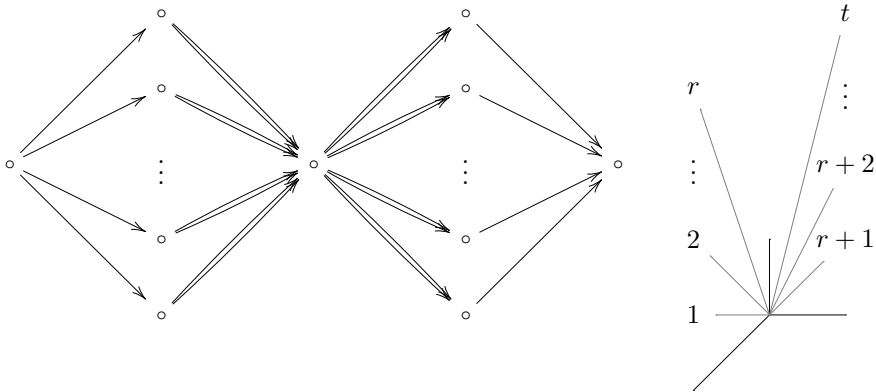

\begin{equation*}
\xygraph{!{0;<2cm,0cm>:<0cm,1cm>::}
[] *+[o]+@{o}="1"
[ruu] *+[o]+@{o}="2"
[d] *+[o]+@{o}="3"
[d] *+{\vdots}
[d] *+[o]+@{o}="4"
[d] *+[o]+@{o}="5"
[uur] *+[o]+@{o}="6"
[ruu] *+[o]+@{o}="7"
[d] *+[o]+@{o}="8"
[d] *+{\vdots}
[d] *+[o]+@{o}="9"
[d] *+[o]+@{o}="10"
[ruu] *+[o]+@{o}="11"
"1":@/-1mm/"2"
"1":@/-1mm/"3"
"1":@/-1mm/"4"
"1":@/-1mm/"5"
"2":@/^0.1em/"6"
"2":@/_0.1em/"6"
"3":@/^0.1em/"6"
"3":@/_0.1em/"6"
"4":@/^0.1em/"6"
"4":@/_0.1em/"6"
"5":@/^0.1em/"6"
"5":@/_0.1em/"6"
"7":@/-1mm/"11"
"8":@/-1mm/"11"
"9":@/-1mm/"11"
"10":@/-1mm/"11"
"6":@/^0.1em/"7"
"6":@/_0.1em/"7"
"6":@/^0.1em/"8"
"6":@/_0.1em/"8"
"6":@/^0.1em/"9"
"6":@/_0.1em/"9"
"6":@/^0.1em/"10"
"6":@/_0.1em/"10"
[rr]="12"
[l(0.5)d]="13"
[r(0.5)uu]="14"
[r(0.5)d]="15"
[u] *+[o]+{r + 1}="16"
[u] *+[o]+{r + 2}="17"
[u] *+{\vdots}
[u] *+[o]+{t}="18"
[ldddd] *+[o]+{1}="19"
[u] *+[o]+{2}="20"
[u] *+{\vdots}
[u] *+[o]+{r}="21"
"12"-"13"
"12"-"14"
"12"-"15"
"12"-@[grey]"16"
"12"-@[grey]"17"
"12"-@[grey]"18"
"12"-@[grey]"19"
"12"-@[grey]"20"
"12"-@[grey]"21"
}
\end{equation*}
\caption{Quiver and fan associated to a toric system coming from two times blowing up $\mathbb{P}^2$.}
\label{p2blowup2}
\end{figure}
The quiver arises from
the quiver of figure \ref{p2quiver} by inserting $r$ additional vertices, corresponding to the
divisors $R_1, \dots, R_t$, between the first and second vertex, and $t - r$ vertices between
the second and third.
The strongly exceptional sequence associated to this toric system is
\begin{equation*}
\sh{O}_X,  \sh{O}(R_r), \dots, \sh{O}(R_1), \sh{O}(H), \sh{O}(2H - R_{r + 1}), \dots, \sh{O}(2H - R_t).
\end{equation*}

Similarly, we can write down a strongly exceptional toric system for the case where $X$ comes from blowing
up a Hirzebruch surface two times:
\begin{equation*}
R_r, R_{r - 1} - R_r, \dots, R_1 - R_2, P - R_1, sP + Q, P - R_{r + 1}, R_{r + 1} - R_{r + 2}, \dots, R_{t - 1} - R_t, R_t, -(a + s)P + Q - \sum_{i = 1}^t R_i
\end{equation*}
with strongly exceptional sequence
\begin{equation*}
\sh{O}_X,  \sh{O}(R_r), \dots, \sh{O}(R_1), \sh{O}(P), \sh{O}((s + 1)P + Q), \sh{O}(P), \sh{O}((s + 2)P + Q- R_{r + 1}), \dots, \sh{O}((s + 2)P + Q - R_t).
\end{equation*}
Figure \ref{fablowup} shows the corresponding quiver and fan.
\begin{figure}[htbp]
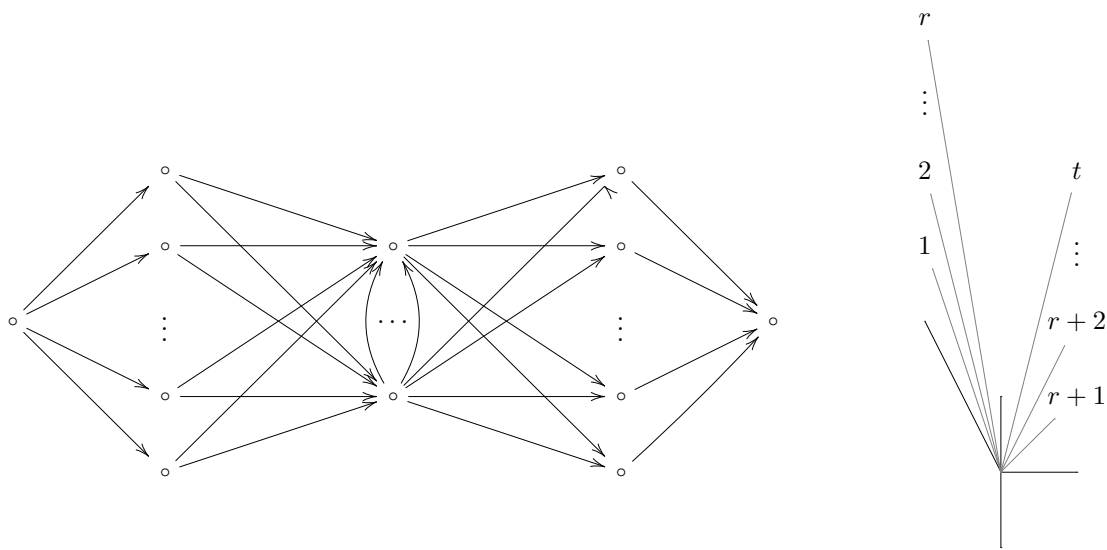

\begin{equation*}
\xygraph{!{0;<2cm,0cm>:<0cm,1cm>::}
[] *+[o]+@{o}="1"
[ruu] *+[o]+@{o}="2"
[d] *+[o]+@{o}="3"
[d] *+{\vdots}
[d] *+[o]+@{o}="4"
[d] *+[o]+@{o}="5"
[r(1.5)u] *+[o]+@{o}="6"
[u] *+{\dots}
[u] *+[o]+@{o}="6b"
[r(1.5)u] *+[o]+@{o}="7"
[d] *+[o]+@{o}="8"
[d] *+{\vdots}
[d] *+[o]+@{o}="9"
[d] *+[o]+@{o}="10"
[ruu] *+[o]+@{o}="11"
"1":@/-1mm/"2"
"1":@/-1mm/"3"
"1":@/-1mm/"4"
"1":@/-1mm/"5"
"2":@/-0.1em/"6"
"3":@/-0.1em/"6"
"4":@/-0.1em/"6"
"5":@/-0.1em/"6"
"2":@/-0.1em/"6b"
"3":@/-0.1em/"6b"
"4":@/-0.1em/"6b"
"5":@/-0.1em/"6b"
"6":@/^1em/"6b"
"6":@/_1em/"6b"
"6":@/-1mm/"7"
"6":@/-1mm/"8"
"6":@/-1mm/"9"
"6":@/-1mm/"10"
"6b":@/-1mm/"7"
"6b":@/-1mm/"8"
"6b":@/-1mm/"9"
"6b":@/-1mm/"10"
"7":@/-0.1em/"11"
"8":@/-0.1em/"11"
"9":@/-0.1em/"11"
"10":@/_0.1em/"11"
[r(1.5)dd]="12"
[u]="15"
[dd]="13b"
[l(0.5)uuu]="13"
[u] *+[o]+{1}="19"
[u] *+[o]+{2}="20"
[u] *+{\vdots}
[u] *+[o]+{r}="21"
[rdddddd]="14"
[u] *+[o]+{r + 1}="16"
[u] *+[o]+{r + 2}="17"
[u] *+{\vdots}
[u] *+[o]+{t}="18"
"12"-"13"
"12"-"13b"
"12"-"14"
"12"-"15"
"12"-@[grey]"16"
"12"-@[grey]"17"
"12"-@[grey]"18"
"12"-@[grey]"19"
"12"-@[grey]"20"
"12"-@[grey]"21"
}
\end{equation*}
\caption{Quiver and fan associated to a toric system coming from two times blowing up $\mathbb{F}_a$.}
\label{fablowup}
\end{figure}
The quiver arises from the
quiver shown in figure \ref{fa2} by inserting $r$ additional vertices corresponding to the
divisors $R_1, \dots, R_r$ between the first and second vertex, and $t - r$ vertices
corresponding to $R_{r + 1}, \dots, R_t$ between the third and fourth vertex. Note that
to maintain the additivity property of the Euler characteristic (Lemma \ref{acyclicitylemma})
we also have to insert
additional arrows from the first layer of new vertices to the former third vertex and
from the former second vertex to the second layer of new vertices.
We conclude:

\begin{theorem}[\cite{HillePerling08}]
Any smooth complete rational surface which can be obtained by blowing up 
a Hirzebruch surface two times (in possibly several points in each step) has a full strongly exceptional
sequence of invertible sheaves.
\end{theorem}

Note that there is no loss of generality, as $\mathbb{P}^2$ blown up in one point is isomorphic to
$\mathbb{F}_1$.
To create above strongly exceptional toric systems, we have made use of the ``slots'' in the toric systems
on $X_0$. So, as there are at most four such slots (for $X_0 \cong \mathbb{F}_a$), it is conceivable
that we can expect by using standard augmentations only to get sequences if $X$ comes from blowing
up $X_0$ simultaneously at most four times. But note that we have used $A_n$ in above construction
to collect all the cohomological ``dirt'' coming from the augmentations. This does not leave too much degree
of freedom for further augmentations. Indeed, it turns out that above type of toric system is maximal:

\begin{theorem}[\cite{HillePerling08}]
Let $\mathbb{P}^2 \neq X$ be a smooth complete rational surface which admits a full strongly
exceptional sequence whose associated toric system is a standard augmentation. Then $X$ can
be obtained by blowing up a Hirzebruch
surface two times (in possibly several points in each step).
\end{theorem}

So far it is not clear, whether every strongly exceptional toric system is related to a standard augmentation.
We will see in the next section that this at least is true for toric surfaces.
Conjecturally, this is true for any rational surface.

\section{A toric example and counterexample}

As mentioned before, on toric surfaces it suffices just to consider standard augmentations:

\begin{theorem}[\cite{HillePerling08}]
Let $X$ be a smooth complete toric surface, then every full strongly exceptional sequence of invertible
sheaves comes from a toric system which is a standard augmentation.
\end{theorem}

Note that we suppress here some technical details which are hidden in the formulation ``comes from''.
For a given toric system we may need to perform a certain normalization process. This process is
trivial but it may be one reason why the general problem of constructing exceptional sequences
of invertible sheaves was perceived as somewhat unwieldy.
An immediate consequence is the following:

\begin{theorem}[\cite{HillePerling08}]
Let $\mathbb{P}^2 \neq X$ be a smooth complete toric surface. Then there exists a full strongly
exceptional sequence of invertible sheaves on $X$ if and only if $X$ can be obtained from a Hirzebruch
surface in at most two steps by blowing up torus fixed points.
\end{theorem}

We can immediately give a bound for toric surfaces which admit a full strongly exceptional 
sequence of invertible sheaves:

\begin{corollary}
Let $X$ be a smooth complete toric surface which admits a full strongly exceptional sequence
of invertible sheaves. Then $\rk \pic(X) \leq 14$.
\end{corollary}

\begin{proof}
We have $\rk \pic(\mathbb{F}_a) = 2$; moreover, $\mathbb{F}_a$, as a toric variety,  has four torus
fixed points. Its blow-up in these four points has 8 fixed points. Hence we can blow up at most
twelve points and the bound follows.
\end{proof}

On the other hand, a Picard number smaller or equal $14$ does not imply that given toric surface
admits such a sequence. Indeed, there is only the trivial bound:

\begin{corollary}
Let $X$ be a smooth complete toric surface with $\rk \pic(X) \leq 4$. Then $X$ admits a strongly
exceptional sequence of invertible sheaves.
\end{corollary}

\begin{proof}
If $\rk \pic(X) \leq 4$, then either $X \cong \mathbb{P}^2$ or $X$ is obtained by blowing up some
$\mathbb{F}_a$ at most two times.
\end{proof}

In \cite{HillePerling06}, a counterexample with Picard number $5$ has explicitly
been verified. The following generalizes this example:

\begin{corollary}[\cite{HillePerling06}]\label{counterexample}
A toric surface given by the following fan for $k \neq 0, 1$ does not admit a full strongly exceptional sequence
of invertible sheaves:
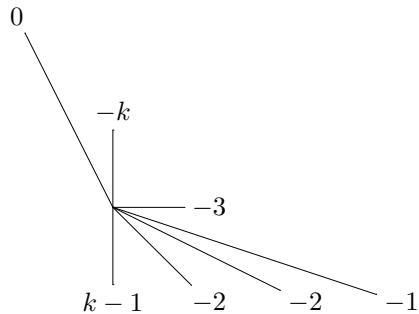
\begin{equation*}
\xygraph{
[]="1"
[d] *+{k - 1}="2"
[luuu] *+{0}="3"
[rd] *+{-k}="4"
[rd] *+{-3}="5"
[d] *+{-2}="6"
[r] *+{-2}="7"
[r] *+{-1}="8"
"1"-"2"
"1"-"3"
"1"-"4"
"1"-"5"
"1"-"6"
"1"-"7"
"1"-"8"
}
\end{equation*}
(note that the integers denote the self-intersection numbers of the corresponding toric prime divisors).
\end{corollary}

We leave it to the reader to write down other counterexamples of Picard number $5$.
A funny observation is that the non-existence of a full strongly exceptional sequence does not imply
the non-existence on blow-ups. The following example is a blow-up of a toric variety of the type
as in Corollary \ref{counterexample} for $k = 2$.

\begin{corollary}
The toric surface described by the following fan admits a full strongly exceptional sequence of
invertible sheaves:
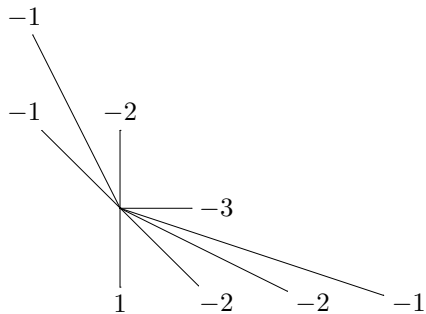
\begin{equation*}
\xygraph{
[]="1"
[d] *+{1}="2"
[luu] *+{-1}="3"
[u] *+{-1}="3b"
[rd] *+{-2}="4"
[rd] *+{-3}="5"
[d] *+{-2}="6"
[r] *+{-2}="7"
[r] *+{-1}="8"
"1"-"2"
"1"-"3"
"1"-"3b"
"1"-"4"
"1"-"5"
"1"-"6"
"1"-"7"
"1"-"8"
}
\end{equation*}
\end{corollary}

To see this, we have only to give an appropriate sequence of simultaneous blow-downs:
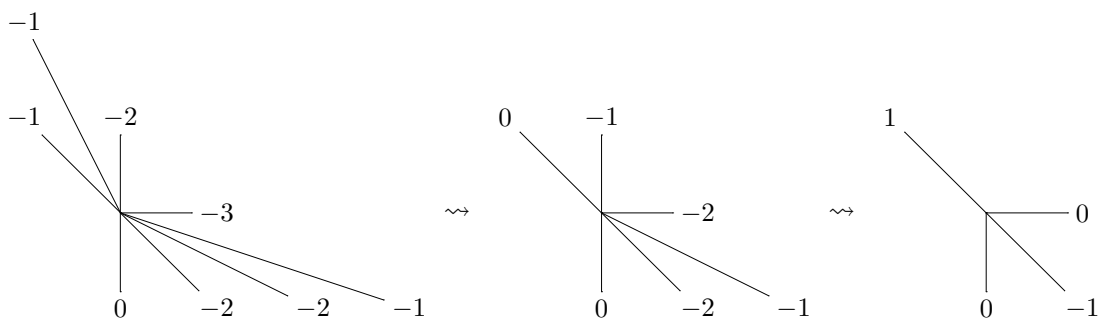
\begin{equation*}
\xygraph{
[]="1"
[d] *+{0}="2"
[luu] *+{-1}="3"
[u] *+{-1}="3b"
[rd] *+{-2}="4"
[rd] *+{-3}="5"
[d] *+{-2}="6"
[r] *+{-2}="7"
[r] *+{-1}="8"
[r(0.5)u] *+{\rightsquigarrow}
"1"-"2"
"1"-"3"
"1"-"3b"
"1"-"4"
"1"-"5"
"1"-"6"
"1"-"7"
"1"-"8"
[urr]="a"
[d] *+{0}="b"
[luu] *+{0}="c"
[r] *+{-1}="d"
[rd] *+{-2}="e"
[d] *+{-2}="f"
[r] *+{-1}="g"
[r(0.5)u] *+{\rightsquigarrow}
"a"-"b"
"a"-"c"
"a"-"d"
"a"-"e"
"a"-"f"
"a"-"g"
[urr]="p"
[d] *+{0}="q"
[luu] *+{1}="r"
[rrd] *+{0}="s"
[d] *+{-1}="t"
"p"-"q"
"p"-"r"
"p"-"s"
"p"-"t"
}
\end{equation*}

\begin{remark}\label{lazyremark}
It would possibly be of interest for some readers to compare the explicit computations in \cite{HillePerling06}
with the methods of \cite{HillePerling08}. In particular, the methods for determine cohomology vanishing
on toric varieties have significantly been improved. Unfortunately, providing the setup for this is beyond
the scope of this note. Therefore we have to refer the reader to \cite{HillePerling08}.
\end{remark}

\section{Cyclic strongly exceptional sequences}\label{cyclicsection}

Consider a toric system $A_1, \dots, A_n$.
So far, we have considered the element $A_n$ in a toric system as a placeholder on which we could
conveniently dump cohomologies from augmentation. However, it is reasonable to ask whether there
can be a statement for strongly exceptional sequences analogous to the fact that if the toric system
gives rise to an exceptional sequence, then this is true for every cyclic renumbering of the $A_i$.
This leads in a natural way to the following definition:

\begin{definition*}
An infinite sequence of sheaves $\dots, \sh{E}_i, \sh{E}_{i + 1}, \dots $ is called a (full) {\em cyclic
strongly exceptional sequence} if $\sh{E}_{i + n} \cong \sh{E}_i \otimes \sh{O}(-K_X)$ for
every $i \in \Z$ and if every  subinterval $\sh{E}_{i + 1}, \dots, \sh{E}_{i + n}$
forms a (full) strongly exceptional sequence.
\end{definition*}

Our notion of cyclic strongly exceptional sequences is very close to the concept of a helix as
developed in \cite{Rudakov90}, and in particular to the geometric helices of
\cite{BondalPolishchuk}. The difference is that we do not  require that cyclic strongly exceptional
sequences are generated by mutations. By a result of Bondal
and Polishchuk, the maximal periodicity of a geometric helix on a surface is $3$, which implies
that $\mathbb{P}^2$ is the only rational surface which admits a full geometric helix. Our weaker notion
admits a bigger class of surfaces, but still imposes very strong conditions. First note that if
$A_1, \dots, A_n$ corresponds to a cyclic strongly exceptional sequence, then we have $h^k(\pm A_i) = 0$
for all $k > 0$ and all $1 \leq i \leq n$. This implies by Lemma \ref{acyclicitylemma} and Theorem \ref{structuretheorem}
that for the associated toric surface $Y$, the associated prime divisors $D_1, \dots, D_n$ have
self-intersection numbers $\geq -2$. It is a well-known fact of toric geometry that this in turn implies that
$-K_Y$ is nef. There are actually only $16$ such toric surfaces, which are shown in table \ref{weakdelpezzofigure}.
\begin{table}[htbp]
\centering
\begin{tabular}{c|c}
Name & self-intersection numbers $D_1^2, \dots, D_n^2$ \\ \hline
$\mathbb{P}^2$ & 1, 1, 1 \\
$\mathbb{P}^1 \times \mathbb{P}^1$ & 0, 0, 0, 0 \\
$\mathbb{F}_1$ & 0, 1, 0, -1 \\
$\mathbb{F}_2$ & 0, 2, 0, -2 \\
5a & 0, 0, -1, -1, -1 \\
5b & 0, -2, -1, -1, 1 \\
6a & -1, -1, -1, -1, -1, -1 \\
6b & -1, -1, -2, -1, -1, 0 \\
6c & 0, 0, -2, -1, -2, -1 \\
6d & 0, 1, -2, -1, -2, -2 \\
7a & -1, -1, -2, -1, -2, -1, -1 \\
7b & -1, -1, 0, -2, -1, -2, -2 \\
8a & -1, -2, -1, -2, -1, -2, -1, -2 \\
8b & -1, -2, -2, -1, -2, -1, -1, -2 \\
8c & -1, -2, -2, -2, -1, -2, 0, -2 \\
9 & -1, -2, -2, -1, -2, -2, -1, -2, -2 \\
\end{tabular}
\caption{The 16 complete smooth toric surfaces whose anticanonical divisor is nef.}\label{weakdelpezzofigure}
\end{table}
So, if any surface $X$ admits a cyclic strongly exceptional sequence, then by Theorem \ref{structuretheorem}
its Picard group is isomorphic to the Picard group of one of the toric surfaces in table \ref{weakdelpezzofigure}.
In particular, we get:

\begin{theorem}[\cite{HillePerling08}]
Let $X$ be a smooth complete rational surface which admits a full cyclic strongly exceptional sequence
of invertible sheaves. Then $\rk \pic(X) \leq 7$.
\end{theorem}

This is a rather strong bound which implies that not even every del Pezzo surface admits such a sequence. As
in Corollary \ref{counterexample}, a small Picard number does not necessarily imply the existence of a full
cyclic strong exceptional sequence. However, there are the following existence results:

\begin{theorem}[\cite{HillePerling08}]
Let $X$ be a del Pezzos surface with $\rk \pic(X) \leq 7$, then there exists a full cyclic
strongly exceptional sequence of invertible sheaves on $X$.
\end{theorem}

\begin{theorem}[\cite{HillePerling08}]
Let $X$ be a smooth complete toric surface, then there exists a full cyclic strongly exceptional sequence
of invertible sheaves on $X$ if and only if $-K_X$ is nef.
\end{theorem}

A particular interest in cyclic strongly exceptional sequences comes from the fact that 
the total space $\pi: \omega_X \rightarrow X$ of the canonical bundle
is a local Calabi-Yau manifold. It follows from results of Bridgeland \cite{Bridgeland05} that
a full strongly exceptional sequence $\sh{E}_1, \dots, \sh{E}_n$ on $X$ is cyclic iff the pullbacks
$\pi^*\sh{E}_1, \dots, \pi^*\sh{E}_n$ form a strongly exceptional sequence on $\omega_X$
(note that the sheaves $\pi^*\sh{E}_i$ are not simple, so that on $\omega_X$ we have to
consider a slightly different notion for exceptional sequences, see \cite{Bridgeland05}).
The direct sum $\bigoplus_{i = 1}^n \pi^*\sh{E}_i$
represents a tilting sheaf, whose associated endomorphism algebra in general is not finite. It can
be represented by a quiver with relations which has loops and contains the quiver of $\oplus_i \sh{E}_i$ as a subquiver.
Table \ref{smallcyclic} shows all possible cyclic strongly exceptional toric systems and their quivers for
varities with Picard numbers $1$ and $2$.
\begin{table}[htbp]
\centering
\begin{tabular}{|>{\PBS\centering\hspace{0pt}}m{1.3cm}|>{\PBS\centering\hspace{0pt}}m{3.3cm}|>{\PBS\centering\hspace{0pt}}m{3cm}|} \hline
surface & toric system & quiver\\ \hline
$\mathbb{P}^2$ & H, H, H &
$
\xygraph{
[] *+[o]+@{o}="1"
[rr] *+[o]+@{o}="2"
[lu] *+[o]+@{o}="3"
"1":@/^0.5em/"2"
"1":"2"
"1":@/_0.5em/"2"
"2":@/^0.5em/"3"
"2":"3"
"2":@/_0.5em/"3"
"3":@/^0.5em/"1"
"3":"1"
"3":@/_0.5em/"1"
}
$ \\ \hline
$\mathbb{P}^1 \times \mathbb{P}^1$ & P, Q, P, Q
&
$
\xygraph{
[] *+[o]+@{o}="1"
[r] *+[o]+@{o}="2"
[u] *+[o]+@{o}="3"
[l] *+[o]+@{o}="4"
"1":@/^0.2em/"2"
"1":@/_0.2em/"2"
"2":@/^0.2em/"3"
"2":@/_0.2em/"3"
"3":@/^0.2em/"4"
"3":@/_0.2em/"4"
"4":@/^0.2em/"1"
"4":@/_0.2em/"1"
}
$
\\ \hline
$\mathbb{P}^1 \times \mathbb{P}^1$ & P, P + Q, P, -P + Q &
$
\xygraph{
[] *+[o]+@{o}="1"
[r] *+[o]+@{o}="2"
[u] *+[o]+@{o}="3"
[l] *+[o]+@{o}="4"
"1":@/^0.2em/"2"
"1":@/_0.2em/"2"
"1":@/^0.2em/"4"
"1":@/_0.2em/"4"
"2":@/^0.2em/"3"
"2":@/_0.2em/"3"
"4":@/^0.2em/"3"
"4":@/_0.2em/"3"
"3":@/^0.3em/"1"
"3":@/_0.3em/"1"
"3":@/^0.1em/"1"
"3":@/_0.1em/"1"
}
$
\\ \hline
$\mathbb{F}_1$ & P, Q, P, -P + Q &
$
\xygraph{
[] *+[o]+@{o}="1"
[r] *+[o]+@{o}="2"
[u] *+[o]+@{o}="3"
[l] *+[o]+@{o}="4"
?"1":@/^0.2em/"2"
"1":@/_0.2em/"2"
"2":@/-0.2em/"3"
"3":@/^0.2em/"4"
"3":@/_0.2em/"4"
"4":@/^0.2em/"1"
"4":@/_0.2em/"1"
"4":@/-0.3em/"1"
"1":@/-0em/"3"
"2":@/-0.1em/"4"
}
$
\\ \hline
$\mathbb{F}_2$ & P, -P + Q, P, -P + Q &
$
\xygraph{
[] *+[o]+@{o}="1"
[r] *+[o]+@{o}="2"
[u] *+[o]+@{o}="3"
[l] *+[o]+@{o}="4"
"1":@/^0.2em/"2"
"1":@/_0.2em/"2"
"2":@/^0.2em/"3"
"2":@/_0.2em/"3"
"3":@/^0.2em/"4"
"3":@/_0.2em/"4"
"4":@/^0.2em/"1"
"4":@/_0.2em/"1"
"1":@/_0.2em/"3"
"3":@/_0.2em/"1"
"2":@/_0.2em/"4"
"4":@/_0.2em/"2"
}
$
\\ \hline
\end{tabular}
\caption{The cyclic strongly exceptional sequences for Picard numbers $1$ and $2$.}\label{smallcyclic}
\end{table}
These are well-known and their relations can be found elsewhere (\cite{Bondal90}, \cite{King2}, \cite{perling6}).

We can see in table \ref{smallcyclic} that $\mathbb{P}^1 \times \mathbb{P}^1$ admits two cyclic strongly
exceptional toric systems, whose associated toric surfaces are isomorphic to $\mathbb{P}^1 \times \mathbb{P}^1$
and $\mathbb{F}_2$, respectively. On the other hand, $\mathbb{F}_2$ admits only one such toric system,
and its associated toric surface is isomorphic to $\mathbb{P}^1 \times \mathbb{P}^1$. This phenomenon occurs
also for other toric weak del Pezzo surfaces, where more symmetric cases admit more types of cyclic strongly
exceptional toric systems than
the less symmetric ones. Using toric methods (which we will not discuss here, see Remark \ref{lazyremark}), it is a
somewhat tedious but straightforward exercise to check all possible standard augmentations for the $16$ cases.
The result is shown in table \ref{weakdelpezzocyclic}.
\begin{table}[htbp]
\centering
\begin{tabular}{ l | l}
surface & type of toric system \\ \hline
$\mathbb{P}^2$ & $\mathbb{P}^2$\\
$\mathbb{P}^1 \times \mathbb{P}^1$ & $\mathbb{P}^1 \times \mathbb{P}^1$, $\mathbb{F}_2$ \\
$\mathbb{F}_1$ & $\mathbb{F}_1$ \\
$\mathbb{F}_2$ & $\mathbb{P}^1 \times \mathbb{P}^1$ \\
5a & 5a, 5b \\
5b & 5a \\
6a & 6a, 6b, 6c, 6d \\
6b & 6a, 6b, 6c \\
6c & 6a, 6b \\
6d & 6a \\
7a & 7a, 7b \\
7b & 7a \\
8a & 8a, 8b, 8c \\
8b & 8a,  8b \\
8c & 8a \\
9 & 9 \\
\end{tabular}
\caption{The table shows which toric weak del Pezzo surfaces can appear as cyclic strongly exceptional toric system on another
toric surface.}\label{weakdelpezzocyclic}
\end{table}

\section{Some noncommutative resolutions of singularities}

There are precisely five toric weak del Pezzos surfaces such that the space $\omega_X$ provides a crepant
resolution of a quotient singularity. More precisely, in these cases, we have a diagram\footnote{we assume in this
section that our base field is $\C$}:
\begin{equation*}
\xymatrix{
& \omega_X \ar[r] \ar[d] & X \\
\C^3 \ar[r] & \C^3 / G,
}
\end{equation*}
where the vertical morphism is the contraction of the base of $\omega_X$ and $G$ a finite abelian subgroup
of $\operatorname{SL}_3(\C)$. In these five case we can construct cyclic strongly exceptional toric systems
such that the corresponding quiver coincides with the McKay quiver of the action of $G$ on $\C^3$.
\begin{table}
 \centering
\begin{tabular}{c c c}
\begin{tabular}{c}
$\mathbb{P}^2$\\
$\xygraph{
[] *+[o]+@{o}="1"
[rr] *+[o]+@{o}="2"
[lu] *+[o]+@{o}="3"
"1":@/^0.5em/"2"
"1":"2"
"1":@/_0.5em/"2"
"2":@/^0.5em/"3"
"2":"3"
"2":@/_0.5em/"3"
"3":@/^0.5em/"1"
"3":"1"
"3":@/_0.5em/"1"
}$\\[1mm]
$\C^3 / \Z_3$
\end{tabular}&
\begin{tabular}{c}
$\mathbb{F}_2$ \\
$\xygraph{
[] *+[o]+@{o}="1"
[r] *+[o]+@{o}="2"
[u] *+[o]+@{o}="3"
[l] *+[o]+@{o}="4"
"1":@/^0.2em/"2"
"1":@/_0.2em/"2"
"2":@/^0.2em/"3"
"2":@/_0.2em/"3"
"3":@/^0.2em/"4"
"3":@/_0.2em/"4"
"4":@/^0.2em/"1"
"4":@/_0.2em/"1"
"1":@/_0.2em/"3"
"3":@/_0.2em/"1"
"2":@/_0.2em/"4"
"4":@/_0.2em/"2"
}
$\\[1mm]
$\C^3 / \Z_4$
\end{tabular}&
\begin{tabular}{c}
6d \\
$\xygraph{
[] *+[o]+@{o}="1"
[rd] *+[o]+@{o}="2"
[r] *+[o]+@{o}="3"
[ur] *+[o]+@{o}="4"
[lu]*+[o]+@{o}="5"
[l]*+[o]+@{o}="6"
"1":@/-0em/"2"
"1":@/-0em/"3"
"1":@/_0.2em/"4"
"2":@/-0em/"3"
"2":@/-0em/"4"
"2":@/_0.2em/"5"
"3":@/-0em/"4"
"3":@/-0em/"5"
"3":@/_0.2em/"6"
"4":@/-0em/"5"
"4":@/-0em/"6"
"4":@/_0.2em/"1"
"5":@/-0em/"6"
"5":@/-0em/"1"
"5":@/_0.2em/"2"
"6":@/-0em/"1"
"6":@/-0em/"2"
"6":@/_0.2em/"3"
}
$\\[1mm]
$\C^3 / \Z_6$
\end{tabular}\\
\begin{tabular}{c}
8c\\
$\xygraph{
[] *+[o]+@{o}="1"
[ld] *+[o]+@{o}="2"
[d] *+[o]+@{o}="3"
[rd] *+[o]+@{o}="4"
[uuur] *+[o]+@{o}="5"
[rd] *+[o]+@{o}="6"
[d] *+[o]+@{o}="7"
[ld] *+[o]+@{o}="8"
"1":@/-0em/"2"
"1":@/-0em/"6"
"1":@/^0.1em/"7"
"2":@/-0em/"3"
"2":@/-0em/"7"
"2":@/^.1em/"8"
"3":@/-0em/"4"
"3":@/-0em/"8"
"3":@/^0.1em/"5"
"4":@/-0em/"1"
"4":@/-0em/"5"
"4":@/^0.1em/"6"
"5":@/-0em/"6"
"5":@/-0em/"2"
"5":@/^0.1em/"3"
"6":@/-0em/"7"
"6":@/-0em/"3"
"6":@/^0.1em/"4"
"7":@/-0em/"8"
"7":@/-0em/"4"
"7":@/^0.1em/"1"
"8":@/-0em/"5"
"8":@/-0em/"1"
"8":@/^0.1em/"2"
}$\\[1mm]
$\C^3 / \Z_2 \times \Z_4$
\end{tabular}&
\begin{tabular}{c}
9\\
$\xygraph{
[] *+{\ }
[r(2.37)] *+[o]+@{o}="1"
[d(0.3)l] *+[o]+@{o}="2"
[l(0.75)d] *+[o]+@{o}="3"
[r(0.25)d(1.2)] *+[o]+@{o}="4"
[d(0.8)r(0.9)] *+[o]+@{o}="5"
[r(1.2)] *+[o]+@{o}="6"
[u(0.8)r(0.9)] *+[o]+@{o}="7"
[r(0.25)u(1.2)] *+[o]+@{o}="8"
[l(0.75)u] *+[o]+@{o}="9"
"1":@/-0em/"2"
"1":@/-0em/"5"
"1":@/-0em/"8"
"2":@/-0em/"3"
"2":@/-0em/"6"
"2":@/-0em/"9"
"3":@/-0em/"4"
"3":@/-0em/"7"
"3":@/-0em/"1"
"4":@/-0em/"5"
"4":@/-0em/"8"
"4":@/-0em/"2"
"5":@/-0em/"6"
"5":@/-0em/"9"
"5":@/-0em/"3"
"6":@/-0em/"7"
"6":@/-0em/"1"
"6":@/-0em/"4"
"7":@/-0em/"8"
"7":@/-0em/"2"
"7":@/-0em/"5"
"8":@/-0em/"9"
"8":@/-0em/"3"
"8":@/-0em/"6"
"9":@/-0em/"1"
"9":@/-0em/"4"
"9":@/-0em/"7"
}$\\[1mm]
$\C^3 / \Z_3 \times \Z_3$
\end{tabular}
\end{tabular}
\caption{The five Mckay quivers coming from cyclic strongly exceptional toric systems on toric surfaces.}\label{mckaytable}
\end{table}
Table \ref{mckaytable} shows the five cases and their corresponding quivers. It turns out that the associated
algebras are isomorphic to the skew group algebras $\C[x,y,z] * G$ with respect to the induced action of $G$
on $\C[x,y,z]$. To exemplify this, consider the case of $\C^3 / \Z_4$. The group $\Z_4$ acts with weights
$(1,1,2)$ on $\C^3$. The quiver with relations is given by:
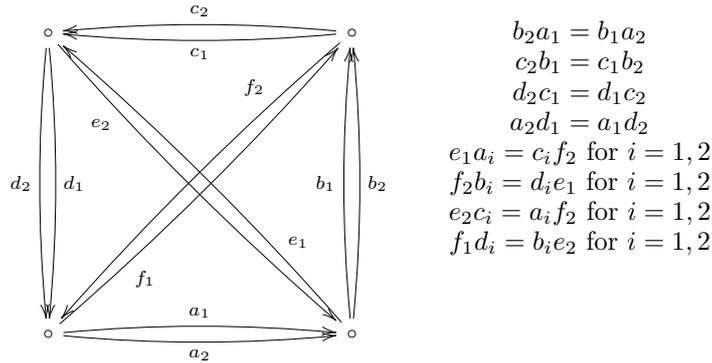
\begin{equation*}
\xygraph{!{0;<2cm,0cm>:<0cm,2cm>::}
[] *+[o]+@{o}="1"
[rr] *+[o]+@{o}="2"
[uu] *+[o]+@{o}="3"
[ll] *+[o]+@{o}="4"
[r(3.5)] *+\txt{$b_2 a_1 = b_1 a_2$}
[d(0.2)] *+\txt{$c_2 b_1 = c_1 b_2$}
[d(0.2)] *+\txt{$d_2 c_1 = d_1 c_2$}
[d(0.2)] *+\txt{$a_2 d_1 = a_1 d_2$}
[d(0.2)] *+\txt{$e_1 a_i = c_i f_2$ for $i = 1, 2$}
[d(0.2)] *+\txt{$f_2 b_i = d_i e_1$ for $i = 1, 2$}
[d(0.2)] *+\txt{$e_2 c_i = a_i f_2$ for $i = 1, 2$}
[d(0.2)] *+\txt{$f_1 d_i = b_i e_2$ for $i = 1, 2$}
"1":@/^0.3em/"2"^{a_1}
"1":@/_0.3em/"2"_{a_2}
"2":@/^0.3em/"3"^{b_1}
"2":@/_0.3em/"3"_{b_2}
"3":@/^0.3em/"4"^{c_1}
"3":@/_0.3em/"4"_{c_2}
"4":@/^0.3em/"1"^{d_1}
"4":@/_0.3em/"1"_{d_2}
"1":@/_0.3em/"3"_(0.25){f_1}
"3":@/_0.3em/"1"_(0.25){f_2}
"2":@/_0.3em/"4"_(0.25){e_1}
"4":@/_0.3em/"2"_(0.25){e_2}
}
\end{equation*}
One compares this directly with the quiver corresponding to $\C[x,y,z] * \Z_4$
(see \cite{CrawMaclaganThomas07b}, Remark 2.7). Skew group
algebras are examples for noncommutative crepant resolutions in the sense of van den Bergh
\cite{vandenBergh04a}, \cite{vandenBergh04b}.

\comment{

\begin{table}[htbp]
\centering
\begin{tabular}{|c|c|}\hline
$\chi(D)$ & $D$\\ \hline
$0$ & $R_1 - R_2$, $R_2 - R_1$, $Q - P - R_1$\\ \hline
$1$ & $R_1$, $R_2$, $P - R_1$, $P - R_2$, $Q - P$, $Q - R_1 - R_2$\\ \hline
$2$ & $P$, $Q - R_1$, $Q - R_2$ \\ \hline
$3$ & $Q$, $Q + P - R_1 - R_2$ \\ \hline
$4$ & $2Q - R_1 - R_2$, $P + Q - R_1$, $P + Q - R_2$ \\ \hline
$5$ & $2Q - R_1$, $P + 2Q - 2 R_1 - R_2$, $2Q - R_2$, $2P + Q - R_1 - R_2$ \\ \hline
$6$ & $2Q$, $P + 2Q - 2 R_1$, $2P + Q - R_1$, $2P + Q - R_2$, $3Q - 2 R_1 - R_2$ \\ \hline
families & $Q + nP$ for $n \geq -1$, \\
& $Q + nP - R_1$ for $n \geq -1$, \\
& $Q + nP - R_2$ for $n \geq 0$ \\
& $Q + n(Q - R_1)$ for $n \geq 0$ \\
& $P + n(Q - R_1)$ for $n \geq 0$\\ \hline
\end{tabular}
\caption{Strongly leftorthogonal divisors with Euler characteristic $\leq 3$ on the variety 9}\label{totallyacyclic9}
\end{table}
}

\section{Commutative and Non-commutative deformations}\label{deformationsection}

Tilting correspondence provides a way to study noncommutative deformations of algebraic
varieties via deformations of their derived categories. We will give two straightforward
examples of such
deformations by means of deformations of the endomorphism algebras of tilting sheaves,
i.e. if $A$ is such an algebra and $A_t$ a deformation of this algebra with respect to some
parameter $t$, then $D^b(A_t-\operatorname{mod})$ is considered as deformation of
$D^b(A-\operatorname{mod})$. However, we will not consider any formal setup for this
kind of deformation, so one should more carefully speak of a mere parametrization
of certain derived categories. We will see below (Theorem \ref{p2blowup}) that this
kind of parametrization is compatible with well-known geometric deformations.
In a similar spirit, deformations have been studied
in the context of homological mirror symmetry (e.g. see \S 2 of \cite{AKO2} ).

Consider first the quiver with relations as shown in figure \ref{fanoncom}.
\begin{figure}[htbp]
\begin{equation*}
\xygraph{!{0;<1.5cm,0cm>:<0cm,2cm>::}
[] *+[o]+@{o}="1"
[rr] *+[o]+@{o}="2"
[r] *+{\vdots}
[r] *+[o]+@{o}="3"
[rr] *+[o]+@{o}="4"
"1":@/^0.3em/"2"^{a_1}
"1":@/_0.3em/"2"_{a_2}
"2":@/^1.5em/"3"^{l_0}
"2":@/_1.5em/"3"_{l_{k + 1}}
"3":@/^0.3em/"4"^{d_1}
"3":@/_0.3em/"4"_{d_2}
}
\end{equation*}
\begin{align*}
p_i (l_i a_1 &- l_{i + 1} a_2) \text{ for } 0 \leq i < k  & %r_i ( d_1 l_i a_1 &- d_2 l_{i + 2} a_2) \text{ for } 0 \leq i < k
q_i (d_1 l_i &- d_2 l_{i + 1}) \text{ for } 0 \leq i < k \\
\end{align*}
\caption{A path algebra whose relations depend on parameters $p_i, q_i$.}
\label{fanoncom}
\end{figure}
Here, we have a quiver with relations which depend on parameters $p_i, q_i$ from our base
field. For fixed $k$ we denote $\mathcal{M}_k$ the $3k$-dimensional parameter space of parameters
$p_i, q_i, r_i$ as above.

\begin{theorem}
The parameter space $\mathcal{M}_k$ contains the path algebras associated to the toric systems of the
form $P, sP + Q, P, -(s + a) P + Q$ on $\mathbb{F}_a$, where $a = k - 1 - 2s$ and $0 \leq s < \frac{k + 1}{2}$.
\end{theorem}

\begin{proof}
For some integer $0 \leq s < \frac{k + 1}{2}$ we just consider the limit $p_s, q_s%, r_{s - 1}, r_s
\longrightarrow 0$, and $p_i, q_i%, r_i
\longrightarrow 1$ for all other parameters,
we obtain precisely the quiver with relations corresponding to the strongly exceptional toric system
$P, sP + Q, P, -(s + a) P + Q$ on $\mathbb{F}_a$, where $a = k - 1 - 2s$, as in figure \ref{fa2}.
\end{proof}

This way, via the derived category, we have found a noncommutative deformation space which allows
to deform surfaces $\mathbb{F}_a$ and $\mathbb{F}_b$  into each other for $0 \leq a, b < k$ and $b \equiv
a \mod 2$.\footnote{Actually their derived categories. Note that two rational surfaces with equivalent derived
categories are isomorphic (see \cite{Huybrechts06}, \S 12).}
Note that the parametrization presented here is just chosen for simplicity, as, of course,
there are more possibilities for choosing such parameters which specialize as desired.

In more general cases it is more difficult to describe noncommutative deformations via such explicit
parametrizations. We give an example where we embed a parameter space of some rational surfaces
as locally closed subset into a noncommutative parameter space. Consider algebraic surfaces $X$ which can
be obtained by simultaneously blowing up $\mathbb{P}^2$ in $t > 1$ points. By Proposition \ref{blowuponce},
the following is a strongly exceptional toric system on $X$:
\begin{equation*}
R_t, R_{t - 1} - R_t, \dots, R_1 - R_2, H - R_1, H, H - \sum_{i = 1}^t R_i.
\end{equation*}
Using the canonical isomorphism with $\pic(Y)$, where $Y$ is the toric surface associated to the
toric system, we can identify $\pic(X)$ and $\pic(X')$ for any two simultaneous blowups $X$ and
$X'$ of $\mathbb{P}^2$ in $t$ points. In particular, we can consider this toric system as a universal
strongly exceptional toric system for all simultaneous blow-ups of $\mathbb{P}^2$ in $t$ points.

The natural parameter space for these blow-ups is given by the open subset of $(\mathbb{P}^2)^t$
given by the complement of the diagonals. We denote this parameter space by $\mathcal{M}$
and write $X \in \mathcal{M}$ for a rational surface which is a blow-up of $\mathcal{P}^2$ in $t$
distinct points. If one is interested in proper isomorphism classes,
one also has to take the diagonal action of $\operatorname{PGL}_3$ on $\mathcal{M}$
into account. However, we will neglect this aspect for simplicity.

The quiver associated to our toric system is the same for every $X \in \mathcal{M}$ and looks as
shown in figure \ref{p2blowup1}. We denote $A$ the path algebra which corresponds to this
quiver. The algebra which corresponds to our toric system then is of the form $A_X \cong A/I_X$,
where $I_X$ is an ideal of relations which depends on $X$. Both algebras are finite-dimensional:
\begin{equation*}
\dim A = 18t + 6 \quad \text{ and } \quad \dim A_X = \sum_{I} h^0(\sum_{i \in I} A_i) = 9t + 15,
\end{equation*}
where the first sum runs over all intervals $I \subset \{1, \dots, n - 1\}$. Parameter spaces we are
interested in are given by any parametrization of ideals of $A$ which contains the $I_X$ for
$X \in \mathcal{M}$. We want to describe explicitly a map $X \mapsto I_X$ by specifying the
ideal $I_X$ as explicit as possible. This way we will have an embedding of $\mathcal{M}$
into an appropriate parameter space.

For this, we have to relate the vector spaces $e_i A e_j$ and $e_i A_X e_j$, where $e_i, e_j$ are
idempotents of $A$ (and therefore of $A_X$), enumerated as in figure \ref{p2blowup1b}.
We denote $V$ a three-dimensional vector space such that $\mathbb{P}^2 \cong \mathbb{P}V$.
Then:
\begin{align*}
e_{t + 2} A e_i = e_{t + 2} A_X e_i &= \Gamma\big(X, \sh{O}(H - R_i)\big) =: H_i \text{ for } 1 \leq i \leq t\\
e_i A e_{t + 1} = e_i A_X e_{t + 1} &= \Gamma\big(X, \sh{O}(R_i)\big) \text{ for } 1 \leq i \leq t\\
e_{t + 3} A e_{t + 2} = e_{t + 3} A_X e_{t + 2} &= \Gamma\big(X, \sh{O}(H)\big) = \Gamma\big(\mathbb{P}^2, \sh{O}(H)\big) = V^*\\
e_{t + 2} A_X e_{t + 1} &= \Gamma\big(X, \sh{O}(H)\big) = \Gamma\big(\mathbb{P}^2, \sh{O}(H)\big) = V^*\\
e_{t + 3} A_X e_{t + 1} &= \Gamma\big(X, \sh{O}(2H)\big) = \Gamma\big(\mathbb{P}^2, \sh{O}(2H)\big) = S^2V^*\\
e_{t + 3} A_X e_i &= \im\big(H_i \otimes V^* \longrightarrow S^2 V^*\big)\\
e_{t + 3} A e_i & = H_i \otimes V^*\\
e_{t + 2} A e_{t + 1} &= \bigoplus_{i = 1}^t H_i\\
e_{t + 3} A e_{t + 1} &= \big( \bigoplus_{i = 1}^t H_i \big) \otimes V^*.
\end{align*}
Figure \ref{p2blowup1b} indicates the vector spaces as situated in the quiver.
\begin{figure}[htbp]
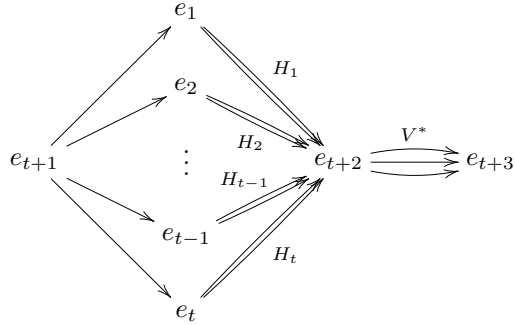

\begin{equation*}
\xygraph{!{0;<2cm,0cm>:<0cm,1cm>::}
[] *+{e_{t + 1}}="1"
[u(0.1)r] *+{\vdots}
[u(0.9)u] *+{e_1}="2"
[d] *+{e_2}="3"
[dd] *+{e_{t - 1}}="4"
[d] *+{e_t}="5"
[uur] *+{e_{t + 2}}="6"
[r] *+{e_{t + 3}}="7"
"1":@/^0.01mm/"2"
"1":@/-0mm/"3"
"1":@/-0mm/"4"
"1":@/-0mm/"5"
"2":@/^0.1em/"6"^{H_1}
"2":@/_0.1em/"6"
"3":@/^0.1em/"6"
"3":@/_0.1em/"6"_{H_2}
"4":@/^0.1em/"6"^{H_{t - 1}}
"4":@/_0.1em/"6"
"5":@/^0.1em/"6"
"5":@/_0.1em/"6"_{H_t}
"6":@/^0.5em/"7"^{V^*}
"6":@/_0.5em/"7"
"6":@/-0mm/"7"
}
\end{equation*}
\caption{Quiver for $A_X$.}
\label{p2blowup1b}
\end{figure}
Note that the map from $H_i$ to $V^*$ is just the natural inclusion of
$\Gamma\big(X, \sh{O}(H - R_i)\big)$ into $\Gamma\big(X, \sh{O}(H)\big)$.
In particular, these inclusions encode the geometric information coming from
$X$: we can identify
$H_i$ with the hyperplane in $V^*$ which is dual to the corresponding
$i$-th point in $\mathbb{P}V$ which gets blown up. We get:

\begin{theorem}\label{p2blowup}
With notation as specified above, the ideal $I_X$ is generated by:
\begin{align*}
& \ker \big( H_i \otimes V^* \longrightarrow S^2 V^*\big) \text{ for } 1 \leq i \leq t,\\
& \ker \big(\bigoplus_{i = 1}^t H_i \longrightarrow V^*\big),\\
& \ker \big(V^* \otimes V^* \longrightarrow S^2V^*).
\end{align*}
\end{theorem}

\begin{proof}
These maps represent the relations among paths of lengths $2$ and $3$ and follow
from the explicit representations of the $e_i A e_j$ and $e_i A_X e_j$ above.
For paths of length $3$ just note that every such path passes through $e_{t + 2}$
and therefore can be represented by an element of $e_{t + 1} A_X e_{t + 2} \otimes
e_{t + 2} A_X e_{t + 3} \cong V^* \otimes V^*$.
\end{proof}

Note that the last relation corresponds to the relations of type
$a_i b_j = a_j b_i$ as for the case of $\mathbb{P}^2$ (see figure \ref{p2quiver}).

\end{document}